\documentclass[10pt,a4paper]{article}

\usepackage{natbib}        
\usepackage{graphicx}      
\usepackage{amsmath}
\usepackage{amssymb}
\usepackage{comment}
\usepackage{mynotation}
\usepackage{algorithm}
\usepackage{tikz}
\usepackage{pgfplots}
\usepackage{amsthm}
\pgfplotsset{compat=newest}
    \newlength\figureheight
    \newlength\figurewidth
\usetikzlibrary{matrix,positioning}
\def\AN{\A_{22}} %

\newtheorem{thm}{Theorem}
\newtheorem{cor}[thm]{Corollary}
\newtheorem{lem}[thm]{Lemma}

\newtheorem{prop}[thm]{Proposition}

\newtheorem{exmp}[thm]{Example}
\newtheorem{rem}[thm]{Remark}

\title{Discretizing stochastic dynamical systems using Lyapunov equations
\thanks{This work is supported by the Swedish Foundation for Strategic Research under the project Cooperative Localization. All three authors are with the Division of Automatic Control,
Link\"oping University, Sweden.  E-mail: nikwa@isy.liu.se, axelsson@isy.liu.se, fredrik@isy.liu.se.}
}

\author{Niklas Wahlstr\"om \and Patrix Axelsson \and Fredrik Gustafsson}

\begin{document}
\maketitle

\begin{abstract}                
Stochastic dynamical systems are fundamental in state estimation, system
identification and control. System models are often provided in
continuous time, while a major part of the applied theory is developed
for discrete-time systems. Discretization of continuous-time models is
hence fundamental.
We present a novel algorithm using a
  combination of Lyapunov equations and analytical solutions, enabling
  efficient implementation in software. The proposed method
  circumvents numerical problems exhibited by standard algorithms in
  the literature. Both theoretical and simulation results are
  provided.
\end{abstract}



\section{Introduction}
Dynamical processes in engineering and physics have for a long time successfully been modeled with continuous-time differential equations. 
In order to account for uncertainties, these equations are usually driven by an unknown stochastic process called process noise. This noise is ideally modeled as completely ``white'' in order to obtain the Markov property, which is required in recursive Bayesian inference, such as Kalman filtering. 
However, in order to implement such filtering, the continuous-time model has to be discretized. 
Such discretization includes solving an integral involving the matrix exponential on the form
\begin{align} \label{eq:int}
\Q = \int_0^T e^{\A \tau}\S e^{\A^\Transp \tau}d\tau,
\end{align}
where $\A, \S, \Q \in \mathbb{R}^{n \times n}$.

We propose an algorithm for solving \eqref{eq:int} by decomposing the problem into subproblems and then solve these subproblems either analytically or using a combination of Lyapunov and Sylvester equations.

In many practical applications the discrete-time process noise covariance is modeled and tuned directly, rather than discretized from its continuous-time counterpart. However, in certain scenarios the dependency between the discrete-time process noise covariance and the sampling time is important. If the filtering should work on different devices with different sampling frequencies, this dependency should be properly modeled to guarantee the same dynamical behavior of the filter. Further, in data with non-equidistant sampling the process noise covariance has to be rescaled at each time instant. This is often the case in Gaussian process regression which can be described  with a state-space model and solved using Kalman filtering \citep{sarkka:13}. 

In the literature there exist different algorithms for computing the integral \eqref{eq:int}. 
The probably most well-cited one was presented by \cite{VanLoan:78}, which involves computing the matrix exponential for an augmented $2n\hspace{-0.5mm}\times\hspace{-0.5mm}2n$ matrix followed by a matrix multiplication of two resulting submatrices. This method does not require any assumption on the model, however the resulting matrix becomes ill-conditioned if the sampling time is large or if the poles of the system are fast.
For certain models, \eqref{eq:int} can be solved analytically. \cite{Rome:69} presented a direct solution under the assumption that $\A$ is diagonalizable. The method requires an eigenvalue decomposition which is not always numerical stable \citep{Higham:08} and not all matrices are diagonalizable. Finally, the integral can always be solved numerically using the trapezoidal or the rectangular method.



In this work we present an alternative method for solving
\eqref{eq:int}. This method is based on a Lyapunov equation which
characterizes the solution of \eqref{eq:int}. However, since Lyapunov
equations cannot be solved if the system contains integrators
\citep{antoulas:2005}, the problem is decomposed into subproblems
where the integrators are treated separately. As will be explained,
one set of subproblems cannot be solved using Lyapunov equations,
but they do have an analytical solution of \eqref{eq:int}. Conversely,
the remaining set of subproblems do not have a closed form solution of
\eqref{eq:int}, but then the method with Lyapunov equations works fine. The
algorithm involves computing the matrix exponential of the
$n\hspace{-0.5mm}\times\hspace{-0.5mm}n$ system matrix rather than an
augmented $2n\hspace{-0.5mm}\times\hspace{-0.5mm}2n$ matrix as
required by the solution by Van Loan. Furthermore, the proposed
algorithm circumvents some numerical problems in the method proposed by Van Loan.
Our theoretical algebraic contributions include:
\begin{itemize}
\item A Lemma describing the relation between \eqref{eq:int} and the aforementioned Lyapunov equation, see Lemma~\ref{thm:Q}.
\item A novel extension of this solution which also handles integrators, see Section~\ref{sec:combined}.
\item A complete algorithm which solves \eqref{eq:int} with complementing numerical properties compared to existing solutions, see Algorithm~\ref{alg:proposed}.
\end{itemize}

The outline of the paper is as follows. In Section~\ref{sec:preliminaries} the mathematical models are presented and the importance of the discretization method in use is motivated. In Section~\ref{sec:lyap} the discretization using Lyapunov equations is presented together with the main theoretical contributions of the paper. In Section~\ref{sec:combined} the solutions from the previous two sections will be combined to solve for systems with integrators. In Section~\ref{sec:num} a numerical evaluation is performed and in Section~\ref{sec:con} the conclusions are summarized and future directions pointed out.



\section{Mathematical preliminaries}
\label{sec:preliminaries}
Consider the following It\^{o} stochastic differential equation 
\begin{subequations}
\begin{align} \label{eq:cont_dyn}
d\x(t) & = \A \x(t)dt + d\vec{\beta}(t),
\end{align}
where $\vec{\beta}(t)$ is a Brownian motion with
\begin{align}
\E[d\vec{\beta}(t) d\vec{\beta}(t)^\Transp] = \S dt.
\end{align}
\end{subequations}
The model \eqref{eq:cont_dyn} is formally equivalent to the stochastic differential equation
\begin{subequations} \label{eq:cont_dyn2}
\begin{align} 
\frac{d\x(t)}{dt} & = \A \x(t) + \w(t), \label{eq:cont_dyn2a}
\end{align}
where $\w(t)$ is a zero-mean white Gaussian process with
\begin{align}
\E[\w(t) \w(\tau)^\Transp] = \S\delta(t-\tau).
\end{align}
\end{subequations}
Since $\w(t)$ is not square Riemann integrable, the model \eqref{eq:cont_dyn2} does not have any mathematical meaning \citep{Jazwinski:70}. However, we can still intuitively think of it as a stochastic differential equation driven by white noise. 

It is important to note that this is just a model of the physical process and cannot be found in reality. For example, white noise has a flat power spectral density requiring infinite power, which is not physically realizable. Nevertheless, using this continuous-time model will lead to sound properties for the equivalent discrete-time model as will be explained later.

By integrating \eqref{eq:cont_dyn} over the time interval $[t_k,t_{k+1}]$ we can find its discrete-time equivalence as
\begin{align}
\x(t_{k+1})  & = \underbrace{e^{\A T_k}}_{\F_{T_k}}\underbrace{\x(t_{k})}_{\x_k} + \underbrace{\int_{t_k}^{t_{k+1}}  e^{\A (\tau-t_{k+1})} d\vec{\beta}(\tau)}_{\w_k},
\end{align}
where $T_k=t_{k+1}-t_k$ is the sampling time. This can be stated as a discrete-time stochastic difference equation
\begin{subequations} \label{eq:disc_mod}
\begin{align} \label{eq:disc_dyn}
\x_{k+1}  & = \F_{T_k} \x_k + \w_k.
\end{align}
By following for example \cite{Jazwinski:70}, the noise $\w_k$ will be zero-mean, white Gaussian
\begin{align}
\E[\w_k \w_{l}^\Transp]=\Q_{T_k}\delta_{kl},
\end{align}
\end{subequations}
where $\delta_{kl}$ is the Kronecker delta function and
\begin{subequations} \label{eq:FQint}
\begin{align} 
\Q_{T_k} & = \int_0^{T_k} e^{\A\tau}\S {e^{\A^\Transp \tau}} d\tau, \label{eq:Qint}
\end{align}
which together with the discrete-time system matrix
\begin{align}
\F_{T_k} & = e^{\A T_k}
\end{align}
\end{subequations}
completes the discretization procedure.

The integral expression \eqref{eq:Qint} can be found in multiple textbooks on Kalman filtering (e.g. \cite{bar:2001,GrewalA:08}) for modeling discrete-time dynamical processes. Nevertheless, the discretization of continuous-time differential equations for filtering applications is often misused. For example, the noise $\w(t)$ is commonly assumed to be constant during each sampling interval leading to the following discrete-time noise covariance
\begin{subequations} \label{eq:alt}
\begin{align} \label{eq:altA}
\bar{\Q}_{T_k}^{\textrm{A}} & = \frac{1}{T_k}\underbrace{\left(\int_0^{T_k} e^{\A\tau}d\tau \right)}_{\bar{\G}_{T_k}}\S \underbrace{\left(\int_0^{T_k} {e^{\A^\Transp \tau}} d\tau\right)}_{\bar{\G}_{T_k}^\Transp},
\end{align}
or just rescaling the continuous-time noise covariance with the sampling time
\begin{align} \label{eq:altB}
\bar{\Q}_{T_k}^{\textrm{B}} & = T_k\S.
\end{align}
\end{subequations}
In contrast to the discretization in \eqref{eq:FQint}, the assumptions in \eqref{eq:alt} lead to a dynamical description of the process which depends on the sampling intervals, whereas the actual physical process do not. This can be seen by the property derived in the following example.
\begin{exmp}
Consider the three time instances $t_1$, $t_2$ and $t_3$. We then have
\begin{subequations}
\begin{align}
\textrm{Cov}\Big[\x(t_{3})\Big|\x(t_1)\Big]
& = \textrm{Cov}\Big[\F_{t_3-t_2}\x(t_{2})+\w_2\Big|\x(t_1)\Big] \\
& = \textrm{Cov}\Big[\F_{t_3-t_2}\Big(\F_{t_2-t_1}\x(t_{1})+\w_1\Big)+\w_2\Big|\x(t_1)\Big] \\
& = \textrm{Cov}\Big[\F_{T_2}\w_1+\w_2\Big|\x(t_1)\Big] \\
& = \F_{T_2}\Q_{T_1}\F_{T_2}^\Transp+\Q_{T_2}.
\end{align}
\end{subequations}
We could also use only one time interval and go from $t_1$ directly to $t_3$ with the sampling time $t_3-t_1=T_1+T_2$, which gives
\begin{subequations}
\begin{align}
\textrm{Cov}\Big[\x(t_{3})\Big|\x(t_1)\Big]
& = \textrm{Cov}\Big[\F_{t_3-t_1}\x(t_{1})+\w_1\Big|\x(t_1)\Big]\\
& = \textrm{Cov}\Big[\w_1\Big|\x(t_1)\Big]
 = \Q_{t_3-t_1} = \Q_{T_1 + T_2}.
\end{align}
\end{subequations}
This gives the relation
\begin{align} \label{eq:correct_sampling}
\Q_{T_1+T_2}=\F_{T_2}\Q_{T_1}\F_{T_2}^\Transp+\Q_{T_2}.
\end{align}
\end{exmp}
Indeed, this property is fulfilled for the discretization presented in \eqref{eq:FQint}.
\begin{lem}
If $\F_{T_k}$ and $\Q_{T_k}$ are computed as described in \eqref{eq:FQint}, then
$$\Q_{T_1+T_2}=\F_{T_2}\Q_{T_1}\F_{T_2}^\Transp+\Q_{T_2}.$$
\end{lem}
\begin{proof}
\begin{align*}
\Q_{T_1+T_2} \hspace{-0.5mm} & = \hspace{-0.5mm}\int_0^{T_1+T_2} e^{\A \tau}\S {e^{\A^\Transp \tau}} d\tau \\
						 & = \hspace{-0.5mm}\int_0^{T_2} e^{\A \tau}\S {e^{\A^\Transp \tau}} d\tau + \int_{T_2}^{T_1+T_2} e^{\A \tau}\S {e^{\A^\Transp \tau}} d\tau \\
						 & = \hspace{-0.5mm}\underbrace{\int_0^{T_2} e^{\A \tau}\S {e^{\A^\Transp \tau}} d\tau}_{\Q_2}
						 + \underbrace{e^{\A T_2}}_{\F_{T_2}}\underbrace{\int_0^{T_1} e^{\A^\Transp \tau}\S {e^{\A^\Transp \tau}} d\tau}_{\Q_{T_1}} \underbrace{{e^{\A^\Transp T_2}}}_{\F_{T_2}^\Transp}\\						 
						 & = \hspace{-0.5mm}\Q_{T_2} + \F_{T_2}\Q_{T_1}\F_{T_2}^\Transp.
\end{align*} 
\end{proof}
With similar calculations we can easily derive the equivalent results for the covariance matrices in \eqref{eq:alt} and conclude that they do not share this property since
\begin{subequations}
\begin{align}
\bar{\Q}^{\textrm{A}}_{T_1+T_2} \hspace{-0.5mm} & = \hspace{-0.5mm} \bar{\Q}^{\textrm{A}}_{T_2} \hspace{-0.5mm} + \hspace{-0.5mm} \F_{T_2}\bar{\Q}^{\textrm{A}}_{T_1}\F_{T_2}^\Transp \hspace{-0.5mm} + \hspace{-0.5mm} \F_{T_2}\bar{\G}_{T_1}\S\bar{\G}_{T_2}^\Transp \hspace{-0.5mm} + \hspace{-0.5mm} \bar{\G}_{T_2}\S\bar{\G}_{T_1}^\Transp \F_{T_2}^\Transp \notag \\
																& \neq \hspace{-0.5mm} \bar{\Q}^{\textrm{A}}_{T_2} \hspace{-0.5mm} + \hspace{-0.5mm} \F_{T_2}\bar{\Q}^{\textrm{A}}_{T_1}\F_{T_2}^\Transp,\\
\bar{\Q}^{\textrm{B}}_{T_1+T_2} & = \hspace{-0.5mm} \bar{\Q}^{\textrm{B}}_{T_2} \hspace{-0.5mm} + \hspace{-0.5mm} \bar{\Q}^{\textrm{B}}_{T_1} \notag \\
																& \neq \hspace{-0.5mm} \bar{\Q}^{\textrm{B}}_{T_2} \hspace{-0.5mm} + \hspace{-0.5mm} \F_{T_2}\bar{\Q}^{\textrm{B}}_{T_1}\F_{T_2}^\Transp.
\end{align}
\end{subequations}

Hence by assuming that the underlying continuous-time model is driven by a continuous-time white process the corresponding discrete-time model has the property that the dynamical description does not depend on the sampling intervals, in contrast to other common discretization procedures. We can therefore see \eqref{eq:disc_mod} and \eqref{eq:FQint} as algebraic relations between $\A$, $\S$, $T_k$, $\F_{T_k}$ and $\Q_{T_k}$ fulfilling the property in \eqref{eq:correct_sampling} without deriving it from its continuous-time counterpart.

The main advantage with the alternative expressions in \eqref{eq:alt} in comparison to \eqref{eq:Qint} is their ease of calculation (especially true for \eqref{eq:altB}). The remaining part of this work will therefore describe how the integral \eqref{eq:Qint} can be solved in an efficient manner with good numerical properties.

\section{Discretization using Lyapunov equations}
\label{sec:lyap}

A method for computing the integral \eqref{eq:Qint} will now be presented. The method will be proposed by requiring the system to be asymptotically stable. Later in this section we will prove that this requirements actually can be relaxed.

\subsection{Proposal of solution}
If the system is asymptotically stable, i.e. if all eigenvalues of $\A$ have negative real part, a stationary covariance will exist and we denote it as
\begin{align} \label{eq:P1}
\textrm{Cov}[\x(t)] = \textrm{Cov}[\x_k]& = \P.
\end{align}
This covariance satisfies the following two Lyapunov equations for the continuous-time model \eqref{eq:cont_dyn} and the discrete-time model \eqref{eq:disc_dyn}, respectively
\begin{subequations} \label{eq:P2}
\begin{align} 
\mat{0} & = \A\P+\P\A^\Transp+\S, \label{eq:riccati} \\
\P 			& = \F_{T_k} \P \F_{T_k}^\Transp + \Q_{T_k}.
\end{align}
\end{subequations}
which gives a structured way of computing $\Q_{T_k}$, as presented in Algorithm~\ref{alg:lyap}.

\begin{algorithm}[h]
\caption{Solution using Lyapunov equation for $\P$} \label{alg:lyap}
The matrices $\A$ and $\S$ and the scalar ${T_k}$ are given. The matrices $\F_{T_k}$ and $\Q_{T_k}$ in \eqref{eq:FQint} can then be computed as
\begin{subequations} \label{eq:FQlyap}
\begin{align} 
\F_{T_k} & = e^{\A {T_k}}, \\
\Q_{T_k} & = \P - \F_{T_k} \P \F_{T_k}^\Transp, \label{eq:Qlyap}
\end{align}
where $\P$ is the solution to the Lyapunov equation
\begin{align}
\A\P+\P\A^\Transp = -\S.
\end{align}
\end{subequations}
\end{algorithm}
This algorithm can also be reformulated such that we do not need to compute $\P$ in an intermediate step.
By multiplying \eqref{eq:Qlyap} with $\A$ from the left and with $\A^\Transp$ from the right, respectively, we get
\begin{subequations}
\begin{align}
\A\Q_{T_k} & = \A\P - \F_{T_k} \A\P \F_{T_k}^\Transp, \label{eq:AQ}\\
\Q_{T_k}\A^\Transp & = \P\A^\Transp - \F_{T_k} \P\A^\Transp \F_{T_k}^\Transp, \label{eq:QA}
\end{align}
\end{subequations}
where the fact that $\F_{T_k}$ and $\A$ commute has been used since 
\begin{align*}
\A \F_{T_k} & = \A (\I+\A + \frac{1}{2}\A^2 \dots)=(\A+\A^2 + \frac{1}{2}\A^3 \dots) \\
						& = (\I+\A + \frac{1}{2}\A^2 \dots)\A = \F_{T_k}\A.
\end{align*}
By adding \eqref{eq:AQ} and \eqref{eq:QA}, we get
\begin{align}
\A\Q_{T_k} + \Q_{T_k}\A^\Transp \hspace{-0.5mm}& = \hspace{-0.5mm}\A\P  - \F_{T_k} \A\P \F_{T_k}^\Transp + \P\A^\Transp - \F_{T_k} \P\A^\Transp \F_{T_k}^\Transp \notag \\
																& = \hspace{-0.5mm}\underbrace{\A\P + \P\A^\Transp}_{-\S} - \F_{T_k}(\underbrace{\A\P + \P\A^\Transp}_{-\S})\F_{T_k}^\Transp \notag \\
																& = \hspace{-0.5mm}-\S + \F_{T_k}\S\F_{T_k}^\Transp.
\end{align}
This gives the following algorithm as presented in Algorithm~\ref{alg:lyap2}.
\begin{algorithm}
\caption{Solution using Lyapunov equation for $\Q_{T_k}$} \label{alg:lyap2}
The matrices $\A$ and $\S$ and the scalar ${T_k}$ are given. The matrices $\F_{T_k}$ and $\Q_{T_k}$ in \eqref{eq:FQint} can then be computed as
\begin{subequations} \label{eq:FQlyap2}
\begin{align} 
\F_{T_k} & = e^{\A {T_k}}
\end{align}
and $\Q_{T_k}$ is the solution to the Lyapunov equation
\begin{align}
\A\Q_{T_k}+\Q_{T_k}\A^\Transp=-\V_{T_k}, \label{eq:Qlyap2}
\end{align}
where
\begin{align}
\V_{T_k}=\S-\F_{T_k}\S\F_{T_k}^\Transp. \label{eq:V}
\end{align}
\end{subequations}
\end{algorithm}
This algorithm is similar to the solution presented by \cite{axelsson:12} derived from a continuous-time differential Lyapunov equation.

From here on we will proceed with Algorithm~\ref{alg:lyap2}. However, all results (including the final algorithm) can be reformulated to suit Algorithm~\ref{alg:lyap} as well.

\subsection{Theoretical result}
It can now be proven that Algorithm~\ref{alg:lyap2} (and consequently also Algorithm~\ref{alg:lyap}) gives a solution to \eqref{eq:FQint}, provided that the solution of the Lyapunov equation exists and is unique.
\begin{lem} \label{thm:Q}
The solution to the integral
\begin{subequations}
\begin{align} \label{eq:int_lem1}
\Q=\int_0^{T}e^{\A \tau}\S e^{\B\tau} d\tau
\end{align}
satisfies the Sylvester equation 
\begin{align} \label{eq:sylv_lem1}
\A\Q+\Q\B = -\S+e^{\A T}\S e^{\B T}.
\end{align}
\end{subequations}
\end{lem}
\begin{proof}
Start with \eqref{eq:sylv_lem1} and replace $\Q$ with the integral \eqref{eq:int_lem1}. This gives
\begin{subequations}
\begin{align}
\A\Q+\Q\B											 & = \int_0^{T}\A e^{\A \tau}\S e^{\B\tau} d\tau+\int_0^{T}e^{\A \tau}\S e^{\B\tau}\B d\tau \notag \\
											 				& = \int_0^{T} \frac{d}{d\tau}[e^{\A \tau}\S e^{\B\tau}] d\tau \\
											 				& = e^{\A \tau}\S e^{\B\tau}\Big|_0^{T}=e^{\A T}\S e^{\B T} - \S.										 				
\end{align}
\end{subequations}
\end{proof}
\begin{rem}
A similar result was presented by \cite{Gawronski:04} in the context of time-limited grammians. However, that result requires $\B = \A^\Transp$ and that all eigenvalues of $\A$ should have negative real part. 
\end{rem}
\begin{rem}
  Note that Lemma~\ref{thm:Q} does not require anything about the
  matrices $\A$ and $\B$. In particular, they do not need to be stable as
  assumed in \eqref{eq:P1} and \eqref{eq:P2}.
  Indeed, the requirements for the Lyapunov equation
  \eqref{eq:Qlyap2} to have a unique solution are milder. This is
  answered by the following proposition, which is given for the more
  general Sylvester equation.
\end{rem}
\begin{prop} \label{prop:sylvester}
The Sylvester equation
\begin{align}
\A\P+\P\B=\C
\end{align}
has a unique solution $\P$ if and only if
\begin{align} \label{eq:prop_req}
\lambda_i(\A)+\lambda_j(\B) \neq 0 \quad \forall i,j.
\end{align}
\end{prop}
For proof, see for example \cite{antoulas:2005}.

For the case where $\B=\A^\Transp$ and with the requirement that $\A$ is stable, the condition \eqref{eq:prop_req} is always fulfilled. By using that observation together with Lemma~\ref{thm:Q} where $T \rightarrow \infty$, we get the following well known results relating the controllability grammian to a Lyapunov equation, which can be found in most textbooks on linear systems, e.g. \cite{rugh:96}.
\begin{cor}
If all eigenvalues of $\A$ have negative real parts, then for each symmetric matrix $\S$ there exists a unique solution of
\begin{subequations}
\begin{align} \label{eq:sylv_lem2}
\A\Q+\Q\A^\Transp = -\S
\end{align}
given by
\begin{align} \label{eq:int_lem2}
\Q=\int_0^{\infty}e^{\A \tau}\S e^{\A^\Transp \tau} d\tau.
\end{align}
\end{subequations}
\end{cor}
According to Proposition~\ref{prop:sylvester} the integral \eqref{eq:Qint} cannot be computed using the Lyapunov equation \eqref{eq:Qlyap2} if $\A$ and $-\A$ have any common eigenvalues. This is especially the case if the system has integrators, which indeed is common in models intended for Kalman filtering. In the next section we will therefore present a solution which handles such systems as well. With this extension almost all systems of interest will be covered, except for the systems which have at least one pair of non-zero poles mirrored in the imaginary axis.

This extension will be performed by decomposing the problem into subproblems where some of these subproblems still can be solved using parts of the Lyapunov equation \eqref{eq:Qlyap2}, whereas the remaining subproblem can be solved analytically using the integral \eqref{eq:Qint}.

\section{Solution for systems with Integrators}
\label{sec:combined}
Consider the case when $\A$ is block triangular consisting of three blocks
\begin{align} \label{eq:block}
\A=
\begin{bmatrix}
\A_{11} & \A_{12} \\
\mat{0} & \A_{22}
\end{bmatrix},
\end{align}
where 
\begin{subequations} \label{eq:req}
\begin{align} 
\lambda_i(\A_{11})+\lambda_j(\A_{11}) & \neq 0 \quad \forall i,j, \label{eq:req1}\\
\lambda_i(\A_{11})+\lambda_j(\A_{22}) & \neq 0 \quad \forall i,j, \label{eq:req2} \\
\lambda_j(\A_{22}) 										& = 0 \quad \forall i,j. \label{eq:req3}
\end{align}
\end{subequations}
In this manner we have partitioned $\A$ such that all zero eigenvalues have been placed in $\A_{22}$ and all remaining non-zero eigenvalues in $\A_{11}$. Many systems do have such block triangular structure, for example if an observer canonical form has been used, see Example~\ref{ex:obs}. If the system does not have that form, an orthogonal transformation can be applied. This transformation can be computed using Schur decomposition and reordering of the eigenvalues \citep{golub:96}. This will also affect the covariance matrix $\S$ as well as $\V_{T_k}$ by considering this transformation as a state transformation, see Appendix~\ref{app:A}.

\subsection{Solution using Lyapunov and Sylvester equations}
According to Lemma~\ref{thm:Q}, the solution of the integral \eqref{eq:Qint} for the block triangular matrix \eqref{eq:block} shall obey the following Lyapunov equation
\begin{align*}
\begin{bmatrix}
\A_{11} &  \A_{12} \\
\mat{0} & \A_{22}
\end{bmatrix}
\hspace{-1mm}
\begin{bmatrix}
\Q_{11}& \Q_{12} \\
\Q_{12}^\Transp  & \Q_{22}
\end{bmatrix}
\hspace{-1mm}
+
\hspace{-1mm}
\begin{bmatrix}
\Q_{11} & \Q_{12} \\
\Q_{12}^\Transp  & \Q_{22}
\end{bmatrix}
\hspace{-1mm}
\begin{bmatrix}
\A_{11}^\Transp & \mat{0} \\
\A_{12}^\Transp & \A_{22}^\Transp
\end{bmatrix}
\hspace{-1mm}
=
\hspace{-1mm}
-
\hspace{-1mm}
\begin{bmatrix}
\V_{11} \hspace{-0mm}& \V_{12} \\
\V_{12}^\Transp  & \V_{22}
\end{bmatrix}\hspace{-0.5mm},
\end{align*}
where $\Q_{T_k}$ and $\V_{T_k}$ have been partitioned in a similar manner as $\A$. Note that the subscript $T_k$ has been omitted from the submatrices in order to make the notation less cluttered. This gives the following set of Lyapunov and Sylvester equations 
\begin{subequations}
\begin{align}
\A_{11}\Q_{11} +  \Q_{11}\A_{11}^\Transp & = -\V_{11} -\A_{12}\Q_{12}^\Transp - \Q_{12}\A_{12}^\Transp, \label{eq:11a} \\
\A_{11}\Q_{12} +  \Q_{12}\A_{22}^\Transp & = -\V_{12} -\A_{12}\Q_{22 }, \label{eq:12a}\\
\A_{22}\Q_{12}^\Transp +  \Q_{12}^\Transp\A_{11}^\Transp & = -\V_{12}^\Transp -\Q_{22 }\A_{12}^\Transp,\label{eq:21a}\\
\A_{22}\Q_{22} +  \Q_{22}\A_{22}^\Transp & = -\V_{22}. \label{eq:22a}
\end{align}
\end{subequations}

Based on the requirements in \eqref{eq:req1} and \eqref{eq:req2},
Proposition~\ref{prop:sylvester} guarantees that $\Q_{11}$ and
$\Q_{12}$ can be solved uniquely using \eqref{eq:11a} and
\eqref{eq:12a} if $\Q_{22}$ is known.  In contrast, \eqref{eq:22a} does
not have a unique solution for $\Q_{22}$.  Instead,
$\Q_{22}$ can be solved analytically using the integral
\eqref{eq:Qint}. Note that \eqref{eq:21a} is just a transposed version
of \eqref{eq:12a} and does not bring any extra information.

\subsection{Analytical solution for the nilpotent part}
Due to the block triangular structure of  $\A$, the submatrix $\Q_{22}$ will only depend on $\A_{22}$ and $\S_{22}$ via a similar expression as in \eqref{eq:Qint}. By starting from \eqref{eq:Qint} we have
\begin{subequations} \label{eq:Q1}
\begin{align}
\Q_{22} &  = \begin{bmatrix} \mat{0} && \I \end{bmatrix} \Q \begin{bmatrix} \mat{0} \\ \I \end{bmatrix} \\
				& = \begin{bmatrix} \mat{0} && \I \end{bmatrix} \int_0^{T_k} e^{\A\tau}\S {e^{\A^\Transp \tau}} d\tau \begin{bmatrix} \mat{0} \\ \I \end{bmatrix} \\
				& = \int_0^{T_k} \begin{bmatrix} \mat{0} && e^{\A_{22}\tau} \end{bmatrix} \S \begin{bmatrix}  \mat{0} \\e^{\A_{22}^\Transp \tau}  \end{bmatrix} d\tau  \\
& = \int_0^{T_k} e^{\AN\tau}\S_{22} {e^{\AN^\Transp \tau}} d\tau.	\label{eq:Q1last}
\end{align}
\end{subequations}
Further, since all eigenvalues of $\A_{22}$ are zero, the submatrix $\A_{22}$ will also be nilpotent \citep{lancaster:85} leading to
\begin{align}
e^{\AN\tau}=
\sum_{i=0}^{p-1} \AN^i\frac{\tau^i}{i!},
\end{align}
where $p$ is the dimension of $\AN$, i.e. the number of integrators in the system. Expression \eqref{eq:Q1last} can then be computed analytically as
\begin{subequations} \label{eq:Q2}
\begin{align}
\Q_{22} & = \int_0^{T_k} \left(\sum_{i=0}^{p-1} \frac{1}{i!}\AN^i \tau^i \right)\S_{22} \left(\sum_{j=0}^{p-1} \frac{1}{j!}{\AN^j}^\Transp \tau^j \right) d\tau \notag \\
				& = \sum_{i=0}^{p-1} \sum_{j=0}^{p-1} \frac{1}{i!j!}\AN^i \S_{22} {\AN^j}^\Transp  \int_0^{T_k} \tau^{i+j} d\tau \\
				& = \sum_{i=0}^{p-1} \sum_{j=0}^{p-1} \frac{T_k^{i+j+1}}{i!j!(i+j+1)}\AN^i \S_{22} {\AN^j}^\Transp.
\end{align}
\end{subequations}
This is illustrated with the following example.
\begin{exmp}
Consider a constant velocity model, which formally can be described on the form
\begin{align*}
\dot{\x}(t)=
\underbrace{
\begin{bmatrix}
0 && 1 \\
0 && 0
\end{bmatrix}
}_{\A}
\x(t)+
\underbrace{
\begin{bmatrix}
0 \\
1 
\end{bmatrix}
}_{\B}
q(t), \,\,\, \E[q(t)q(\tau)]=\delta(t-\tau).
\end{align*}
This system has only zero eigenvalues which gives
\begin{subequations}
\begin{align}
\A & = \A_{22} = 
\begin{bmatrix}
0 && 1 \\
0 && 0
\end{bmatrix}, \\
\S & = \S_{22} = \E[\B q (\B q)^\Transp] =
\begin{bmatrix}
0 \\
1
\end{bmatrix}
1
\begin{bmatrix}
0 && 1 \\
\end{bmatrix}
=
\begin{bmatrix}
0 && 0 \\
0 && 1
\end{bmatrix}.
\end{align}

By using this in \eqref{eq:Q1} we get
\begin{align}
\Q_{T_k} & = \S T_k + \S\AN^\Transp\frac{T_k^2}{2}+\AN\S\frac{T_k^2}{2}+\AN\S\AN^\Transp\frac{T_k^3}{3} \notag \\
& =
\begin{bmatrix}
0 && 0 \\
0 && 1
\end{bmatrix}T_k + 
\begin{bmatrix}
0 && 0 \\
1 && 0
\end{bmatrix}
\frac{T_k^2}{2}+
\begin{bmatrix}
0 && 1 \\
0 && 0
\end{bmatrix}
\frac{T_k^2}{2}
+
\begin{bmatrix}
1 && 0 \\
0 && 0
\end{bmatrix}
\frac{T_k^3}{3} \notag \\
& =
\begin{bmatrix}
\frac{T_k^3}{3} && \frac{T_k^2}{2} \\
\frac{T_k^2}{2} && T_k
\end{bmatrix},
\end{align}
\end{subequations}
which is the same result as given by \cite{GrewalA:08}, but derived in a different way.
\end{exmp}

\subsection{General algorithm}
Based on the results in the last section, we can now propose an algorithm for computing the integral \eqref{eq:Qint}, also in the case where $\A$ consists of integrators, 
see Algorithm~\ref{alg:proposed}.

\begin{algorithm}
\caption{Proposed algorithm (for systems with arbitrary number of integrators)}  \label{alg:proposed}
The matrices $\A$ and $\S$ and the scalar ${T_k}$ are given. The matrices $\F_{T_k}$ and $\Q_{T_k}$ in \eqref{eq:FQint} can then be computed as
\begin{enumerate}
\setlength{\itemsep}{1mm}
\item \label{it:tf1}
Transform $\A$ and $\S$ to $\At$ and $\St$ such that $\At$ becomes block triangular
\begin{align*}
\U^{-1}\A\U
\hspace{-1mm}
=
\hspace{-1mm}
\At
\hspace{-1mm}
=
\hspace{-1mm}
\begin{bmatrix}
\At_{11} \hspace{-0.5mm}& \At_{12} \\
\mat{0} \hspace{-0.5mm}& \At_{22}
\end{bmatrix},
\,\,\,\,
\U^{-1}\S\U^{-\Transp}
\hspace{-1mm}
=
\hspace{-1mm}
\St
\hspace{-1mm}
=
\hspace{-1mm}
\begin{bmatrix}
\St_{11} \hspace{-0.5mm}& \St_{12} \\
\St_{12}^\Transp \hspace{-0.5mm}& \St_{22}
\end{bmatrix},
\end{align*}
and with all integrators collected in $\At_{22}$. This can be done with an orthogonal transformation computed using Schur decomposition and reordering of the eigenvalues.

\item \label{it:Ft}
Compute $\Ft_{T_k}=e^{\At T_k}$. 
\item \label{it:Vt}
Compute $\Vt_{T_k}=\St-\Ft_{T_k} \St \Ft_{T_k}^\Transp$.
\item \label{it:Qt}
Compute $$\Qt_{T_k}=\begin{bmatrix} \Qt_{11} & \Qt_{12} \\ \Qt_{12}^\Transp & \Qt_{22} \end{bmatrix}$$  using the following steps:
\begin{enumerate}
\setlength{\itemsep}{1mm}
\item Compute $\Qt_{22}$ by evaluating
\begin{align*} 
\Qt_{22} & = \sum_{i=0}^{p-1} \sum_{j=0}^{p-1} \frac{T_k^{i+j+1}}{i!j!(i+j+1)}\At_{22}^i \St_{22} (\At_{22}^i)^\Transp,
\end{align*}
where $p$ is the number of integrators.
\item Compute $\Qt_{12}$ by solving the Sylvester equation
\begin{align} \label{eq:sylv_alg}
\At_{11}\Qt_{12} +  \Qt_{12}\At_{22}^\Transp & = -\Vt_{12} -\At_{12}\Qt_{22}.
\end{align}
\item Compute $\Qt_{11}$ by solving the Lyapunov equation
\begin{align} \label{eq:lyap_alg}
\At_{11}\Qt_{11} +  \Qt_{11}\At_{11}^\Transp & = -\Vt_{11} -\At_{12}\Qt_{12}^\Transp - \Qt_{12}\At_{12}^\Transp.
\end{align}
\end{enumerate}
\item \label{it:tf2}
Transform $\Ft_{T_k}$ and $\Qt_{T_k}$ back to $\F_{T_k}$ and $\Q_{T_k}$ 
\begin{subequations} \label{eq:trans_alg}
\begin{align}
\F_{T_k} & = \U \Ft_{T_k}\U^{-1}, \\
\Q_{T_k} & = \U \Qt_{T_k}\U^{\Transp}.
\end{align}
\end{subequations}
\end{enumerate}
\end{algorithm}

\begin{rem}
If $\A$ does not have any integrators, Algorithm~\ref{alg:proposed} will collapse to the simpler version in Algorithm~\ref{alg:lyap2}.
\end{rem}
\begin{rem}
In theory, $\U^{-1}=\U^\Transp$ since $\U$ is orthogonal. However, numerical algorithms for computing the Schur decomposition do not make $\U$ completely orthogonal. From a numerical point of view it is therefor a benefit to distinguish between $\U^{-1}$ and $\U^\Transp$.
\end{rem}
\begin{rem}
If $\At_{12}=\mat{0}$ the coupling in \eqref{eq:sylv_alg} and \eqref{eq:lyap_alg} via $\Qt_{12}$ and $\Qt_{22}$ will disappear and they can be solved independently from each other. If this is desired, the transformation in Step 1 can be extended to eliminate $\At_{12}$ by solving an addition Sylvester equation \citep{bavely:1979}. However, such transformation is no longer orthogonal and can be arbitrary ill-conditioned if the non-zero eigenvalues are close to zero.
\end{rem}
\begin{rem}
If the system already has a block triangular structure, Step~\ref{it:tf1} and Step~\ref{it:tf2} in Algorithm~\ref{alg:proposed} can be omitted. This is the case for the observer canonical form as seen in the following short example.
\end{rem}

\pagebreak
\begin{exmp} \label{ex:obs}
Consider a SISO-system of order $n=m+p$ with $m$ non-zero poles and $p$ additional integrators described with a transfer function
\begin{align}
G(s)=\frac{b_1s^{m-1}+b_2s^{m-2}+ \dots + b_{m-1}s+b_m}{s^n + a_1 s^{m-1} + \dots + a_{m-1}s + a_m} \cdot \frac{1}{s^p}.
\end{align}
This system can be described with the observer canonical form \citep{glad:00}
\begin{subequations}
\begin{align}
\dot{\x} &=
\left[
\begin{array}{cccc|cccc}
-a_1 & 1 & \dots & 0 & 0 & 0 & \dots & 0\\
\vdots & \vdots & \ddots & \vdots & \vdots & \vdots &  &  \vdots\\
-a_{m-1} & 0 & \dots & 1 & 0 & 0 & \dots & 0\\
-a_{m} & 0 & \dots & 0 & 1 & 0 & \dots & 0\\ \hline
0			 & 0 & \dots & 0 & 0 & 1 & \dots & 0\\
\vdots & \vdots & & \vdots & \vdots  & \vdots  & \ddots & \vdots \\
0			 & 0 & \dots & 0 & 0 & 0 & \dots & 1\\
0			 & 0 & \dots & 0 & 0 & 0 & \dots & 0\\
\end{array}
\right]
\x
+
\left[
\begin{array}{c}
0 \\
\vdots \\
0 \\
b_1 \\
\vdots \\
\vdots \\
b_m \\
\end{array}
\right]\w, \\
\y & =
\left[
\begin{array}{cccc}
1 &  0 &  \dots &   0
\end{array} 
\right]\x,
\end{align}
\end{subequations}
which can be written more compactly as
\begin{subequations}
\begin{align}
\dot{\x} &=
\left[
\begin{array}{c|c}
\A_{11} & \A_{12} \\ \hline
\mat{0} & \A_{22}
\end{array}
\right]
\x
+
\B\w, \\
\y & =
\begin{bmatrix}
1 & 0 & \dots & 0
\end{bmatrix}\x.
\end{align}
\end{subequations}
This system has by construction the desired block triangular structure. 
\end{exmp}

\section{Numerical evaluation}
\label{sec:num}
In this section the numerical properties of the proposed solution will be compared with a standard solution presented by \cite{VanLoan:78} given in Algorithm~\ref{alg:loan}. 

\begin{algorithm}[ht]
\caption{Van Loan's method}  \label{alg:loan}
The matrices $\A$ and $\S$ and the scalar ${T_k}$ are given. The matrices $\F_{T_k}$ and $\Q_{T_k}$ in \eqref{eq:FQint} can then be computed as
\begin{subequations} \label{eq:FQloan}
\begin{enumerate}
\item
Form the $2n \times 2n$ matrix
\begin{align}
\H=
\begin{bmatrix}
\A & \S \\
\mat{0} & -\A^\Transp
\end{bmatrix}.
\end{align}
\item
Compute the matrix exponential
\begin{align} \label{eq:loan_exp}
\e^{\H T_k} =
\begin{bmatrix}
\M_{11} & \M_{12} \\
\mat{0} & \M_{22}
\end{bmatrix}.
\end{align}
\item
Then $\F_{T_k}$ and $\Q_{T_k}$ are given as
\begin{align}
\F_{T_k} = \M_{11}, \qquad \Q_{T_k} = \M_{12} \M_{11}^\Transp.
\end{align}
\end{enumerate}
\end{subequations}
\end{algorithm}
\vspace{-0.5mm}

\subsection{Implementation aspects}
\vspace{-0.5mm}
In both methods \matlab's built-in function {\tt expm} has been used for computing the matrix exponential. In Step 1 of Algorithm~\ref{alg:proposed} the functions {\tt schur} and {\tt ordschur} have been used for computing the Schur decomposition and the reordering of the eigenvalues. Finally, the Lyapunov and Sylvester equations have been solved using {\tt lyap}.
\vspace{-0.5mm}
\subsection{Simulation results}
\vspace{-0.5mm}
In total 100 systems of order $n=6$ with $m=4$ stable poles and $p=2$ additional integrators are randomly generated. Each system is normalized such that the fastest pole is at distance 1 from the imaginary axis, i.e. $|\text{min}(\text{Re}(\lambda_i))|=1$. An estimate $\hat{\Q}_{T_k}$ is computed using both Algorithm~\ref{alg:proposed} and Algorithm~\ref{alg:loan} with single precision for different values of the sampling time $T_k$. Finally, the error
$$\varepsilon = \|\hat{\Q}_{T_k}-\Q_{T_k}\|_2/\|\Q_{T_k}\|_2$$
is evaluated, where $\Q_{T_k}$ is computed using numerical integration of \eqref{eq:Qint} with double precision, here considered as the true value. The result is presented in Figure~\ref{fig:performance}.

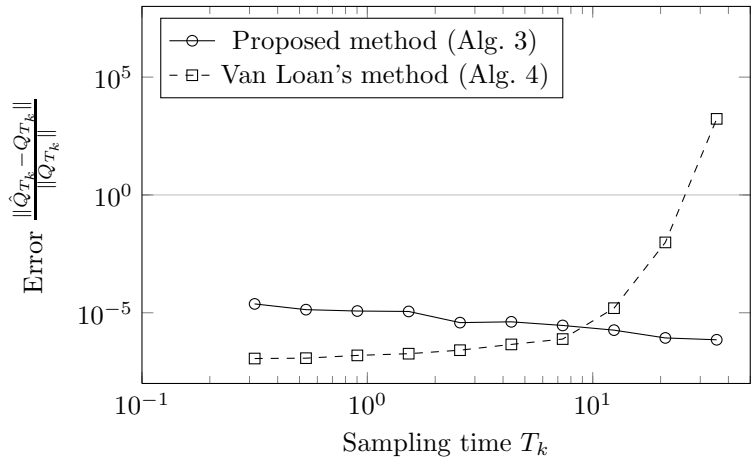
\begin{figure}
\centering
%
%
%
%
\begin{tikzpicture}

\begin{loglogaxis}[%
view={0}{90},
width=\figurewidth,
height=\figureheight,
scale only axis,
xmin=0.1, xmax=50,
xminorticks=true,
xlabel={Sampling time $T_k$},
ymin=1e-08, ymax=1e+08,
ytick={1e-05, 1e00, 1e05},
minor ytick={1e00},
yminorgrids,
ylabel={Error $\frac{\|\hat{\Q}_{T_k}-\Q_{T_k}\|}{\|\Q_{T_k}\|}$},
legend style={at={(0.03,0.97)},anchor=north west,align=left}]
\addplot [
color=black,
solid,
mark=o,
mark options={solid}
]
coordinates{
 (0.316227766016838,2.37607546296204e-05)(0.534290900493161,1.35649788717274e-05)(0.902725177948457,1.185125347547e-05)(1.52522295653902,1.13665391836548e-05)(2.57698037451488,3.82086955141858e-06)(4.35400465365665,4.12989402320818e-06)(7.35642254459641,2.89449053525459e-06)(12.4292362915131,1.83127190211962e-06)(21.0001415570865,8.55957182466227e-07)(35.4813389233575,7.14014902314375e-07) 
};
\addlegendentry{Proposed method (Alg.~\ref{alg:proposed})};

\addplot [
color=black,
dashed,
mark=square,
mark options={solid}
]
coordinates{
 (0.316227766016838,1.14413388985213e-07)(0.534290900493161,1.18521022329787e-07)(0.902725177948457,1.55363338194547e-07)(1.52522295653902,1.83387001584379e-07)(2.57698037451488,2.58392674368224e-07)(4.35400465365665,4.54336714028614e-07)(7.35642254459641,7.66107348226797e-07)(12.4292362915131,1.56696278281743e-05)(21.0001415570865,0.00967796426266432)(35.4813389233575,1690.95483398438) 
};
\addlegendentry{Van Loan's method (Alg.~\ref{alg:loan})};

\end{loglogaxis}
\end{tikzpicture}%
\vspace{-3mm}
\caption{The performance of the proposed method (Algorithm~\ref{alg:proposed}) and Van Loan's method (Algorithm~\ref{alg:loan}).}
\label{fig:performance}
\end{figure}

According to the result the proposed method outperforms the standard method for large $T_k$. The reason will become clear if we investigate the two methods further. In Algorithm~\ref{alg:loan}, both $\A T_k$ and $-\A^\Transp T_k$ are present in the augmented matrix $\H T_k$ and the task to compute its matrix exponential \eqref{eq:loan_exp} will become ill-conditioned if $T_k$ or $|\text{min}(\text{Re}(\lambda_i))|$ is large. In fact, the error will grow exponentially with $T_k$, or the magnitude of work will grow linearly with $T_k$ to keep a certain tolerance \citep{VanLoan:78}. This issue is not present in the proposed method, which can be seen in its simplified version in Algorithm~\ref{alg:lyap}. If $T_k$ is large we have $\F_{T_k}=e^{\A T_k} \approx 0$ and $\Q_{T_k}$ will approach the stationary covariance $\P$ according to \eqref{eq:Qlyap}. The same properties are shared by Algorithm~\ref{alg:proposed}. 

However, for short sampling times the proposed method performers slightly worse. This is especially the case if the system has integrators as well as non-zero poles close to the origin leading to that the Sylvester equation \eqref{eq:sylv_alg} will become ill-conditioned. Future work shall focus on techniques to circumvent this problem. The proposed method has also advantages when it comes to computational complexity since it only needs to compute the matrix exponential of an $n \times n$ matrix rather than of an augmented $2n \times 2n$ matrix as required by van Loan's method.


\section{Conclusions and future work}

\label{sec:con}
An algorithm for computing an integral involving the matrix
exponential common in optimal sampling was proposed. The
algorithm is based on a Lyapunov equation and is justified with
a novel lemma. An extension to systems with integrators was
presented. Numerical evaluations showed that the proposed algorithm has
advantageous numerical properties for large sampling times in
comparison with a standard method in the literature.

Further work includes extending the algorithm further to handle arbitrary matrices, i.e. also matrices with non-zero eigenvalues mirrored in the imaginary axis. Also the numerical properties should be analyzed further and strategies for improving the numerical properties for slow poles should be investigated.

\bibliographystyle{ifacconf-harvard}
\bibliography{refs_ctu}             

\begin{thebibliography}{15}
\expandafter\ifx\csname natexlab\endcsname\relax\def\natexlab#1{#1}\fi
\expandafter\ifx\csname url\endcsname\relax
  \def\url#1{\texttt{#1}}\fi
\expandafter\ifx\csname urlprefix\endcsname\relax\def\urlprefix{URL }\fi

\bibitem[{Antoulas(2005)}]{antoulas:2005}
Antoulas, A.~C., 2005. Approximation of large-scale dynamical systems. {SIAM},
  {Philadelphia, PA, USA}.

\bibitem[{Axelsson and Gustafsson(2012)}]{axelsson:12}
Axelsson, P., Gustafsson, F., 2012. Discrete-time solutions to the
  continuous-time differential {L}yapunov equation with applications to
  {K}alman filtering. Technical report, Link{\"o}ping University, Sweden.

\bibitem[{Bar-Shalom et~al.(2001)Bar-Shalom, Li, and Kirubarajan}]{bar:2001}
Bar-Shalom, Y., Li, X.~R., Kirubarajan, T., 2001. Estimation with Applications
  to Tracking and Navigation: Theory Algorithms and Software. {John Wiley \&
  Sons}, {New York, NY, USA}.

\bibitem[{Bavely and Stewart(1979)}]{bavely:1979}
Bavely, C.~A., Stewart, G.~W., 1979. An algorithm for computing reducing
  subspaces by block diagonalization. SIAM Journal on Numerical Analysis
  16~(2), 359--367.

\bibitem[{Gawronski(2004)}]{Gawronski:04}
Gawronski, W., 2004. Advanced structural dynamics and active control of
  structures. {Springer-Verlag}, {Berlin, Heidelberg, Germany}.

\bibitem[{Glad and Ljung(2000)}]{glad:00}
Glad, T., Ljung, L., 2000. Control theory. {Taylor Francis}, {New York, NY,
  USA}.

\bibitem[{Golub and Van~Loan(1996)}]{golub:96}
Golub, G.~H., Van~Loan, C.~F., 1996. Matrix computations. Vol.~3. {The Johns
  Hopkins University Press}, {Baltimore, MD, USA}.

\bibitem[{Grewal and Andrews(2008)}]{GrewalA:08}
Grewal, M.~S., Andrews, A.~P., 2008. Kalman Filtering. Theory and Practice
  Using \textsc{Matlab}, 3rd Edition. {John Wiley \& Sons}, {Hoboken, NJ, USA}.

\bibitem[{Higham(2008)}]{Higham:08}
Higham, N.~J., 2008. Functions of Matrices -- Theory and Computation. {SIAM},
  {Philadelphia, PA, USA}.

\bibitem[{Jazwinski(1970)}]{Jazwinski:70}
Jazwinski, A.~H., 1970. Stochastic Processes and Filtering Theory. Vol.~64 of
  {Mathematics in Science and Engineering}. {Academic Press}, {New York, NY,
  USA}.

\bibitem[{Lancaster and Tismenetsky(1985)}]{lancaster:85}
Lancaster, P., Tismenetsky, M., 1985. The theory of matrices. Vol.~2. {Academic
  Press}, {New York, NY, USA}.

\bibitem[{Rome(1969)}]{Rome:69}
Rome, H.~J., October 1969. A direct solution to the linear variance equation of
  a time-invariant linear system. {IEEE} Transactions on Automatic Control
  14~(5), 592--593.

\bibitem[{Rugh(1996)}]{rugh:96}
Rugh, W.~J., 1996. Linear System Theory, 2nd Edition. {Information and System
  Sciences Series}. {Prentice Hall Inc.}, {Upper Saddle River, NJ, USA}.

\bibitem[{S{\"a}rkk{\"a} et~al.(2013)S{\"a}rkk{\"a}, Solin, and
  Hartikainen}]{sarkka:13}
S{\"a}rkk{\"a}, S., Solin, A., Hartikainen, J., 2013. Spatio-temporal learning
  via infinite-dimensional {B}ayesian filtering and smoothing. {IEEE} Signal
  Processing Magazine 30~(4), 51--61.

\bibitem[{Van~Loan(1978)}]{VanLoan:78}
Van~Loan, C.~F., June 1978. Computing integrals involving the matrix
  exponential. {IEEE} Transactions on Automatic Control 23~(3), 395--404.

\end{thebibliography}

\appendix
\section{State transformation} \label{app:A}    
\vspace{-0.5mm}
Consider the following state transformation
\vspace{-0.5mm}
\begin{align} \label{eq:trans}
\x=\U\xt.
\end{align}
\vspace{-0.5mm}
By applying \eqref{eq:trans} to the dynamical equation \eqref{eq:cont_dyn2a} we get
\begin{subequations}
\begin{align}
\dot{\x}	& =    \A \x +    \w \quad \Rightarrow \\
\U \dot{\xt}	& =    \A \U \xt +    \w \quad \Rightarrow \\
\dot{\xt}						& = \U^{-1} \A \U \xt + \U^{-1} \w \quad \Rightarrow \\
\dot{\xt}						& = \At \xt + \tilde{\w}.
\end{align}
\end{subequations}
which gives the following transformation of $\A$, $\S$ and $\V$
\begin{subequations} \label{eq:trans1}
\begin{align}
\At & = \U^{-1} \A \U, \\
\St & = \E[\wt \wt^\Transp] = \E[\U^{-1}\w (\U^{-1}\w)^\Transp] = \U^{-1}\E[\w \w^\Transp]\U^{-\Transp} \notag \\
\vspace{-1mm}
		& = \U^{-1} \S\U^{-\Transp}.
\end{align}
\end{subequations}
These matrices will then be used to compute $\tilde{F}_{T_k}$ and $\tilde{Q}_{T_k}$ by following Step \ref{it:Ft}-\ref{it:Qt} in Algorithm~\ref{alg:proposed}. We then have
\begin{subequations}
\begin{align}
\xt_{k+1}	& =  \Ft_{T_k}\xt_{k} +  \wt_k \quad \Rightarrow \\
\U^{-1} \x_{k+1}	& =   \Ft_{T_k}\U^{-1}\x_{k} +  \wt_k \quad \Rightarrow \\
\x_{k+1}					& =   \U \Ft_{T_k}\U^{-1}\x_{k} +  \U\wt_k \quad \Rightarrow \\
\x_{k+1}					& =   \F_{T_k}\x_{k} +  \w_k \quad
\end{align}
\end{subequations}
which gives the transformations
\begin{subequations} \label{eq:trans2}
\begin{align}
\F_{T_k} & = \U \Ft_{T_k}\U^{-1}, \\
\Q_{T_k} & = \E[\w_k \w_k^\Transp] = \E[\U\wt_k (\U\wt_k)^\Transp] = \U\E[\wt_k \wt_k^\Transp]\U^\Transp \notag \\
				 & = \U \Qt_{T_k}\U^{\Transp}.
\end{align}
\end{subequations}
Note that if $\U$ is orthogonal, we have $\U^{-1}=\U^{\Transp}$.
\end{document}